\documentclass{article}
\usepackage{graphicx} 
\usepackage{amsfonts, amsmath,amssymb, amsthm}
\usepackage{todonotes}
\usepackage{authblk}
\newtheorem{theorem}{Theorem}
\newtheorem{lemma}{Lemma}
\usepackage{comment}
\usepackage{natbib}
\usepackage{url}
\usepackage{hyperref}

\usepackage[margin = 1in]{geometry}

\title{Myopic Best Response Replicates the Mean Curvature Flow}

\author[1,2]{John S. McAlister\footnote{Corresponding Author - jm9097@princeton.edu}}
\author[3]{Nathan Burns}

\affil[1]{Princeton University, Program in Applied and Computational Mathematics}
\affil[2]{NSF Center for Analysis and Prediction of Pandemic Expansion (APPEX)}
\affil[3]{University of Tennessee - Knoxville, Department of Mathematics}

\date{}
\begin{document}

\maketitle
\begin{abstract}
    In a continuous coordination game, under myopic best response, players change strategies, and thus strategic communities change their shapes in time. We show that the boundaries of these strategic communities evolve according to Free Boundary Mean Curvature Flow in the continuous time limit. The convergence relies on an approximation scheme for the Mean Curvature Flow, which has previously been used to describe other threshold dynamics. With this equivalence, we can characterize Nash equilibria of the continuous coordination game as minimal surfaces and explore related consequences in biased versions of the same game. 

\end{abstract}

\section{Introduction}\label{sec:introduction}
    In a coordination game, players receive a higher payoff for playing the same strategy as their coplayer(s). This is called the bandwagon property \cite{Kandori1998,cui2022}, and this class of games has been studied for several decades in many different disciplines. In every coordination game, the strategy profile in which all players play the same strategy (regardless of what that strategy is) is a Nash equilibrium. This Nash equilibrium is called the consensus equilibrium \cite{Kandori1993, Ellison1993}. When every player interacts with every other player evenly, this is the only Nash equilibrium, but when the relational structure is not completely connected, there may be Nash equilibria which are not consensus equilibria.  
    
    The case where there is an explicit relational structure among players is called a ``structured coordination game," and it has been studied in discrete settings by many for several decades (e.g., \cite{Ellison1993,Ellison2000,Buskens2016,Oechssler1997,Weidenholzer,Robson1995,Oechssler1999}). These studies have relied on reducing the state space through some symmetric reduction to get analytically tractable results or using simulation to make observations when such a reduction is not possible. Although general analytical results are difficult to find, there are many types of relational structures which are known to only admit consensus equilibria (i.e., $K_n$, $K_{n,m}$ for $n,m$ coprime, $C_n$, $P_n$, and many lattices) \cite{McAlister2024simulation,Ellison1993}. 
    
    The more difficult task is to be able to predict what non-consensus equilibria look like from a general relational structure. It is obvious that finding equilibria is equivalent to locally minimizing the number of edges connecting different players not using the same strategies. In \cite{McAlister2026}, this is shown to have a geometric intuition of minimizing the boundary of the strategic communities (which can be described as a graphical object in the dual sense). This purely geometric description of the game in the discrete setting suggests an analogous relationship between coordination and minimal surfaces in the continuous setting. 

    Interestingly, when we consider this game in the continuous player space, the dynamics of the game become easier to understand, at least in the limit where the radius of interaction for each player is vanishingly small relative to the domain. In this manuscript, we will describe this continuous extension rigorously and show that, if time and space are scaled appropriately, the limit of a Myopic Best Response (MBR) process for this coordination game converges to a limit, and that limit is equivalent to the Free Boundary Mean Curvature Flow (MCF) of the initial strategic boundaries. 

    In section \ref{sec:modelBackground} we give necessary background on the coordination game and on the mean curvature flow. In section \ref{sec:MainResult} we give the main result, that MBR replicates MCF, first in the Cauchy setting, then again in the free boundary setting. We finally explore the implications of the main result and some extensions in section \ref{sec:Implications}.
    
\section{Background}\label{sec:modelBackground}
\subsection{Coordination}\label{subsec:coordination}
    In order to understand this relationship, we must first describe the continuous coordination game in detail. We begin with a standard discrete coordination game where players are vertices $v\in V$ in a graph $G$ which is described by the weighted adjacency matrix $W$. The weights in $W$ describe how frequently one player interacts with another. Each player has an associated strategy from the set of strategies $C$. In the most basic version of the coordination game, where the payoff matrix is the identity matrix, $I_m$, where $m=|C|$, the payoff of any single pairwise interaction is 1 if the two players play the same strategy and 0 otherwise. In the multiplayer game, we say that the payoff is the weighted sum of the payoffs from all possible pairwise interactions. Thus, for a strategy profile $u:V\to C$, the payoff of player $v$ playing strategy $i$ against the strategy profile $u$ is given as
    \begin{equation}\label{eq:DiscretePayoff}w(v,i|u)=\sum_{w\in V}W_{v,w}\delta (i,u(w))\end{equation}
    where $\delta(a,b)=1$ if $a=b$ and $0$ otherwise.

    From this, the extension into continuous space is clear. Now every player is a point in some open domain $\Omega \subseteq \mathbb{R}^n$, so a strategy profile $u$ is a function from $\Omega$ to the set of strategies, $C$, and the weighted adjacency matrix is replaced by a kernel $K(x,y)$ which determines the frequency of interaction between $x$ and $y$. Making these substitutions in \eqref{eq:DiscretePayoff} and replacing the sum with an integral, we get the new payoff function 
    \begin{equation}\label{eq:ContPayoff}w(x,i|u)=\int_\Omega K(x,y)\chi_i(y)dy\end{equation}
    where $\chi_i(y)$ is the characteristic function for the set $\{u(y)=i;y\in \Omega\}.$ 

    This fully characterizes the game. Each player in the set of players, $\Omega$, takes on a strategy from the strategy set $C$ and their payoffs are determined by \eqref{eq:ContPayoff}. The mapping from players to strategies is called a strategy profile, and if no player can improve their payoff by changing strategies, a strategy profile is a Nash equilibrium.  More specifically, because the game is non-atomic \cite{Aumann1964,Schmeidler1973}, (each player is of measure 0 and has no impact on the payoffs of other players), we use the solution concept of an almost everywhere Nash equilibrium, where the Nash condition is satisfied almost everywhere. It is obvious that if (almost) every player is playing the same strategy across the entire domain, then it is a Nash equilibrium. However, we can find examples of non-consensus Nash equilibria. Consider the space $\mathbb{R}^n$ so that the upper half-space $\mathbb{R}^n_+$ plays strategy 1 and the lower half-space and its boundary play strategy 2. This is also a Nash equilibrium for any nonnegative, radially symmetric, measurable kernel. Consider any point in the upper half-space with coordinate $x=(\eta, x')$. Let $S=\{y=(y_1,y')\in \mathbb{R}^n_+; y_1\in (0,2\eta)\}$. Notice that, in the region above $S$, every player plays strategy 1, and in the region below $S$, every player plays strategy 2. Moreover, because $x$ is in the middle of $S$, we know that the mass of $K(x,\cdot)$ below $S$ is equal to the mass of $K(x,\cdot)$ above $S$. In other words, $\int_{\mathbb{R}^n_-}K(x,y)dy=\int_{\mathbb{R}^n_+\setminus S}K(x,y)dy$. This means that   
    \begin{equation*}
    \begin{split} 
    w(x,1|u)&=\int_{\mathbb{R}^n} K(x,y)\chi_1(y)dy\\
    &=\int_SK(x,y)dy+\int_{\mathbb{R}^n_+\setminus S}K(x,y)dy\\
    &=\int_SK(x,y)dy+\int_{\mathbb{R}^n_-}K(x,y)dy\\
    &=\int_SK(x,y)dy+\int_{\mathbb{R}^n_- }K(x,y)\chi_2(y)dy\\
    &\geq w(x,2|u).
    \end{split}
    \end{equation*}

    Thus, no player in the upper half-space could improve their payoff by switching strategies. The same can be said for the lower half-space symmetrically. Those players on the boundary have $w(x,1|u)=w(x,2|u)$ and therefore cannot improve their payoff by switching strategies, no matter which strategy they pick.  This means that no player can improve their payoff unilaterally, and thus we have found a non-consensus Nash equilibrium. 
    
    Although we can treat this game classically as above, to study this game dynamically we must also describe a strategy revision protocol. In this case, we use the standard Myopic Best Response (MBR). MBR is a time-stepping process where, for a strategy profile $u$, every player simultaneously takes on their best response to $u$, resulting in a new strategy profile, $u'$. Player $x$'s best response to $u$ is whichever strategy maximizes their payoff if every other player plays according to the strategy profile $u$. 
    \begin{equation}\label{eq:BRdef}
        BR(x|u)=\arg\max_{i\in C}w(x,i|u)
    \end{equation}

    As an example, consider the strategy profile in figure (Fig. \ref{fig:ex1a}). Computing the payoff of playing strategy $1$ is simply a convolution of the kernel, which in this case is a Gaussian kernel) and $\chi_1$ (Fig. \ref{fig:ex1b}). Once the payoff for each strategy is computed and best responses are determined, a new strategy profile can be drawn from the best responses to the original strategy profile  (Fig. \ref{fig:ex1c}). As players close to the strategic boundary change their strategies, the strategic boundary itself shifts (Fig. \ref{fig:ex1d}). 

    \begin{figure}
        \centering
        \includegraphics[width=0.8\linewidth]{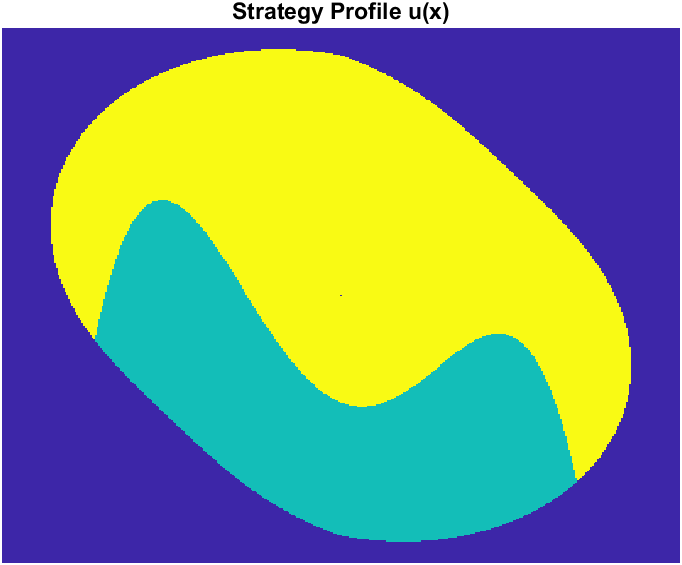}
        \caption{A strategy profile in a bounded domain in $\mathbb{R}^2$. Players colored cyan are playing strategy 1, and players colored yellow are playing strategy 2. Outside of the domain is colored dark blue. }
        \label{fig:ex1a}
    \end{figure}

    \begin{figure}
        \centering
        \includegraphics[width=0.49\linewidth]{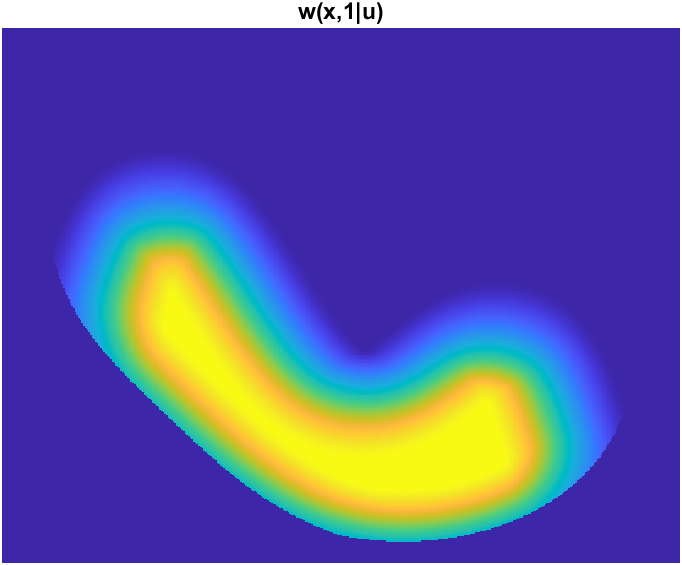}
        \includegraphics[width = 0.49\linewidth]{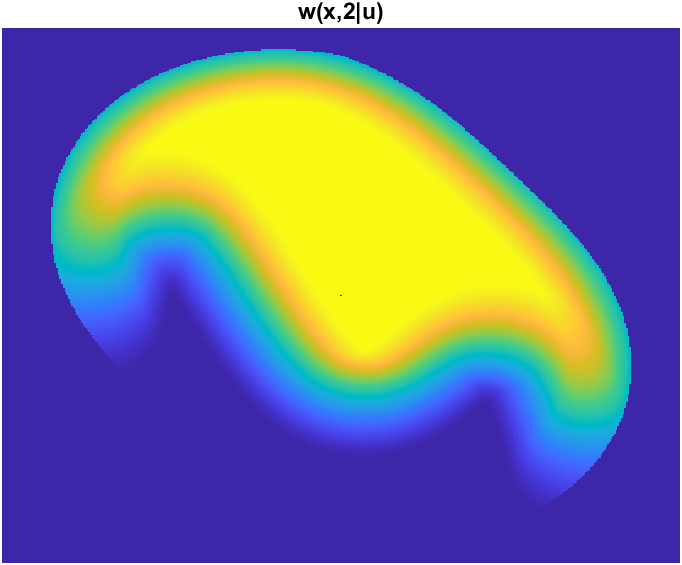}
        \caption{\textbf{left} The payoff of playing strategy 1 against the strategy profile $u$ and \textbf{ right} The payoff of playing strategy 2 against the strategy profile $u$. Brighter colors indicate a higher payoff.}
        \label{fig:ex1b}
    \end{figure}

    \begin{figure}
        \centering
        \includegraphics[width=0.7\linewidth]{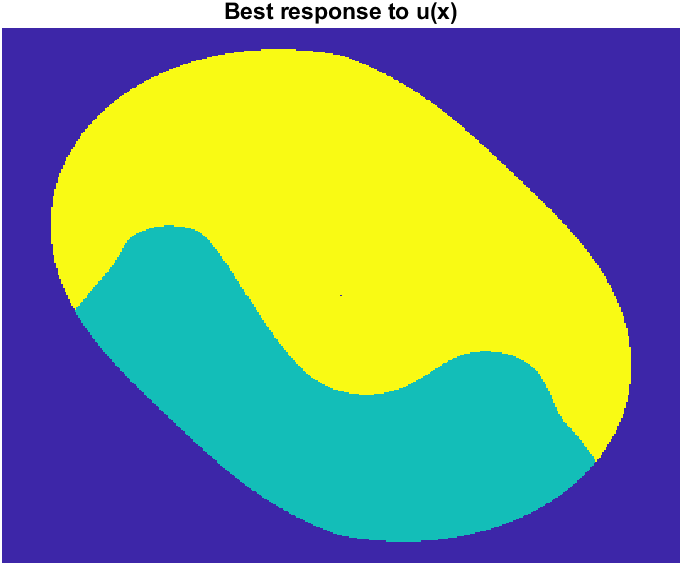}
        \caption{After a single iteration of Myopic best response, the strategy profile from \ref{fig:ex1a} changes. Away from the strategic boundary, most players keep the same strategy, but close to the strategic boundary some players change strategy. This gives the impression that the strategic boundary moves.}
        \label{fig:ex1c}
    \end{figure}

    \begin{figure}
        \centering
        \includegraphics[width=0.7\linewidth]{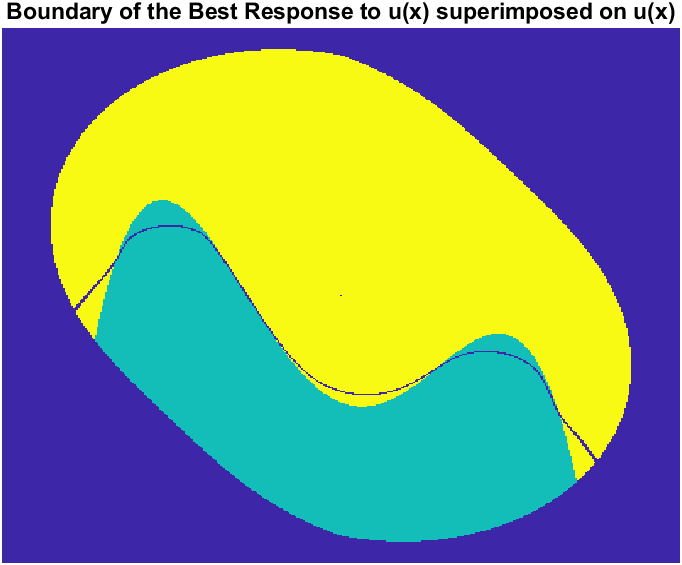}
        \caption{To illustrate the motion of the strategic boundary through myopic best response, we show the original strategy profile $u$ with the strategic boundary from its best response superimposed in black.}
        \label{fig:ex1d}
    \end{figure}
    
    If every player updates their strategy according to their best response simultaneously and repeatedly, this generates a sequence of strategy profiles. 
    \[\{u_t\}_{t=0}^\infty \text{ with }u_t(x)=BR(x|u_{t-1}) \text{ for $t>0$ }\]
    for some initial condition $u_0$. This is a best response sequence, and the boundary of the strategic communities evolves in time (Fig. \ref{fig:ex1e}). It is obvious that, if a best response sequence has a limit, then that limit is necessarily a Nash equilibrium.

    \begin{figure}
        \centering
        \includegraphics[width=0.7\linewidth]{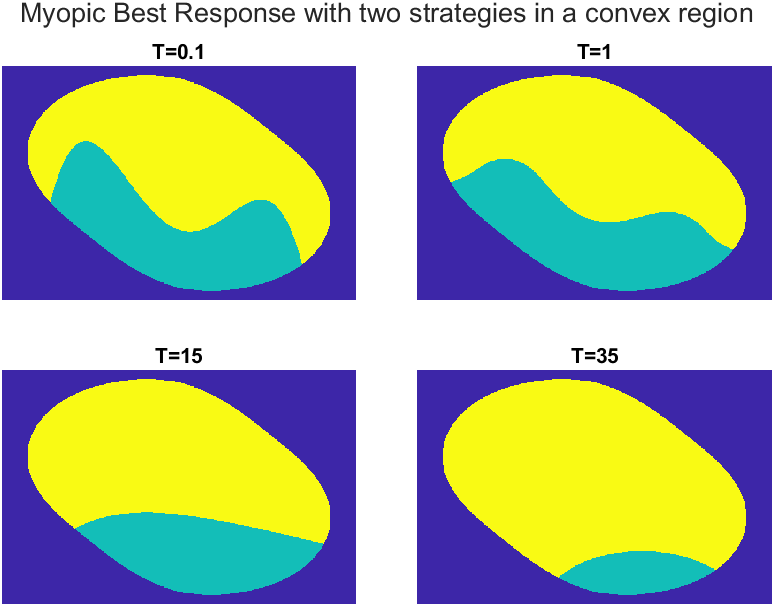}
        \caption{A sequence of strategy profiles generated by myopic best response with an explicit time step. At each time step, every player may change its strategy, and so the boundary between strategic communities shifts.}
        \label{fig:ex1e}
    \end{figure}

    Because the convolution makes the payoff functions uniformly continuous in $\Omega$, the intermediate value theorem ensures that there are players in $\Omega$ that have degenerate best responses, meaning that at least two different strategies maximize payoff. If the set of players with degenerate best responses has positive Lebesgue measure, then the dynamics may be dependent on the choice of tie-breaking rule. However, later in section \ref{sec:MainResult}, through our equivalence, we will be able to argue that the set of players with degenerate best responses will have empty interior and that any tie-breaking order will result in equivalent dynamics, so long as our initial strategy profile is sufficiently well-behaved. 

    It is important to mention that this game is well-posed for an extremely general class of kernels $K$. For the equivalence that we are interested in showing, we must put some additional assumptions on the kernel. These assumptions will be described later in more detail in subsection \ref{subsec:MCF}. However, there are two assumptions with game-theoretic meaning that we must describe now. The first is that the kernel is radially symmetric. This means that the likelihood of two players interacting is dependent only on the distance between these two players. For this reason, we can write $K(x,y)=K(x-y)$ and thus our fitness function is a proper convolution. The second is that it must have an appropriate scale parameter, $\epsilon$. The appropriate scaling that we will require is 
    \begin{equation}\label{eq:scale}K^\epsilon(x-y)=\frac{1}{\epsilon^n}K\left(\frac{x-y}{\epsilon}\right).\end{equation}
    This scaling parameter will be necessary when we embed our discrete sequence that results from MBR into $\Omega\times[0,T)$. The time step, $\Delta t$, will be required to be $\Delta t=\epsilon^2$, which supposes that as the time step decreases, the interaction kernel becomes concentrated close to $0$. This assumption is game-theoretically reasonable because it suggests that with less time to interact between changing strategies, each player will interact more with individuals close to them and less with individuals far away. As an example which we will come back to frequently, the kernel $K^{\sqrt\epsilon}(z)=\frac{1}{(4\pi \epsilon)^{\frac{n}{2}}}\exp\left|\frac{z}{2\sqrt{\epsilon}}\right|^2$ will satisfy all of the conditions for our equivalence and is the well known fundamental solution to the heat equation.  

    Notice that one player's absolute payoff does not depend on any other player's absolute payoff, so the game is not changed by modulating the payoff function by a multiple even if that multiple depends on $x$. This means that the dynamics of the game remain unchanged even if the kernel $K^\epsilon(x-y)$ is replaced with $\tilde K^\epsilon(x-y)=\frac{1}{\|K_\epsilon(x-\cdot)\|_{L^1(\Omega)}}K_\epsilon(x-y)$. In $\mathbb{R}^n$ this has no effect because $\|K^\epsilon(x-\cdot)\|_{L^1(\mathbb{R}^n)}$ is independent of $x$. In a bounded domain, this has the effect of ensuring that $\sum_iw(x,i|u)=1$ for all $x, i,$ and $u$.

    We have now described the game and the strategy revision protocol entirely. Because of the space-time scaling (which is the standard parabolic scaling), we will express a single \textit{time step} of myopic best response with the operator $\mathcal{H}(h)$, so a myopic best response sequence is written as $\{u_0,\mathcal{H}(h)u_0, \mathcal{H}(h)^2u_0,...\}$. We can explicitly express this operator by writing
    \begin{equation}\label{eq:mbr}
        \mathcal{H}(h)u_i = \arg\max_{i\in C}\int_\Omega K^{\sqrt{h}}(x-y)\chi_i(y)dy
    \end{equation} 
    where $\chi_i$ is the indicator function for $\{x\in \Omega;u(x)=i\}$.
    Notice that the kernels are scaled as $\sqrt{h}$ because of the parabolic space-time relationship assumed above. The sequence altogether embedded in space and time is written as \[\bigcup_{n=0}^\infty \{u(nh)\}\times\{nh\} \subset\Omega\times[0,\infty)\] where $u((n+1)h)=\mathcal{H}(h)u(nh)$ and $u(0)=u_0$. For clarity, when we discuss a strategic community $i$, we mean the set of players playing strategy $i$ at time $nh$ (i.e., $Q^i_{nh}=\{x\in \Omega, u(x)=i\}$). Because the boundaries of the strategic communities are crucial to understanding the equivalence we propose, we call the complex of strategic boundaries $\Gamma_{nh}=\bigcup_{i=1}^{m}\partial Q_{n\epsilon}^i\setminus \partial\Omega.$

\subsection{Mean Curvature Flow and its Approximation}\label{subsec:MCF}
Having described the MBR process, we will now give a brief background on the mean curvature flow. Readers interested in a more complete understanding of MCF from a differential geometry perspective should refer to \cite{AndrewsChowGuentherMat2020}. However, the background we present here will be sufficient to understand the equivalence we present and its implications. 

\subsubsection{Motion by mean curvature}

Suppose that we have a closed hypersurface in \((n+1)\)-dimensional space, \(M_{0} \subset \mathbb{R}^{n+1}\), and further suppose we have a variation \(M_{t}\) of \(M_{0}\), that is, a one-parameter family of hypersurfaces which coincides with \(M_{0}\) at \(t = 0\). Then the first-variation formula for the surface area implies 
\[
    \left. \frac{d}{dt} \right|_{t = 0} \text{Vol}(M_{t}) =  \int_{M_{0}}\left\langle H\vec{\nu},\left. \frac{\partial x}{\partial t}\right|_{t = 0}\right\rangle d\mathcal{H}^{n},
\]
where \(x\) is position, \(\mathcal{H}^{n}\) is the usual \(n\)-dimensional Hausdorff measure, \(\vec{\nu}\) is the \emph{outward} pointing normal, and \(H\) is the \emph{mean curvature} defined by \(\text{div}_{M_{0}}\vec{\nu}\). From the first variation of area formula, we see that the most efficient way to instantaneously reduce the volume of \(M_{0}\) is to perturb so that 
\begin{equation}\label{eq:MCF1}
    \displaystyle \left. \frac{\partial{x}}{\partial t}\right|_{t = 0} = -H\vec{\nu}
\end{equation}
on \(M_{0}\). If \eqref{eq:MCF1} is satisfied for each \(t\), then we call the family \(M_{t}\) a \emph{mean curvature flow}. More concretely, we say that a family of hypersurfaces, \(\{M_{t}\}_{t \in [0,T)}\), is a mean curvature flow if 
\begin{equation}\label{eq:MCF}
    \displaystyle\frac{\partial x}{\partial t} = - H\vec{\nu}
\end{equation}
on \(M_{t}\) for each \(t \in [0,T)\). We note that the right-hand side of \eqref{eq:MCF} is equal to \(\Delta_{M_{t}}x\), and so the mean curvature flow equation may be written \((\partial_{t} - \Delta_{M_{t}})x = 0\) and is sometimes called the \emph{geometric heat equation}. An important fact from the evolution equation \eqref{eq:MCF} is that the stationary solutions to the mean curvature flow are precisely those surfaces with mean curvature everywhere equal to \(0\), which are called \emph{minimal surfaces}. If \(n = 1\), we call the resulting flow the \emph{curve shortening flow}. Some examples of the mean curvature flow are:
\begin{itemize}
    \item The stationary plane. In this case \(M_{t} = P\) where \(P\) is a hyperplane in \(\mathbb{R}^{n+1}\), which trivially satisfies \eqref{eq:MCF} since in this case \(H = 0\).
    \item The shrinking sphere. In this case \(M_{t} = \left(\sqrt{r^{2} - 2nt}\right)S^{n}\), which describes a shrinking family of spheres.
    \item The grim reaper. In this case, \(n = 1\) and \(M_{t} = \{(x,y) \in \mathbb{R}^{2} : -\frac{\pi}{2} < x < \frac{\pi}{2}, y = t -\ln\cos x\}\), which describes a family of curves which translate vertically at unit speed.
\end{itemize}

The shrinking sphere demonstrates that singularities may form in finite time, and therefore, it is natural to ask whether or not a singularity may form before the flow disappears. In the case of \(n = 1\), the curve shortening flow (Fig. \ref{fig:csfimage}), the answer is negative in the case that the initial curve is closed and embedded, which is due to the combined work of Gage-Hamilton \cite{GageHamilton1986} and Grayson \cite{Grayson1989}. However, in higher dimensions, closed examples of singularity formation before the flow disappears do exist; see \cite[Remark 3.8]{Ecker2004} for example. In this case, one would like to develop a theory for evolving past singularities, and one such way is the so-called \emph{level-set flow}.

\begin{figure}
    \centering
    \includegraphics[width=0.6\linewidth]{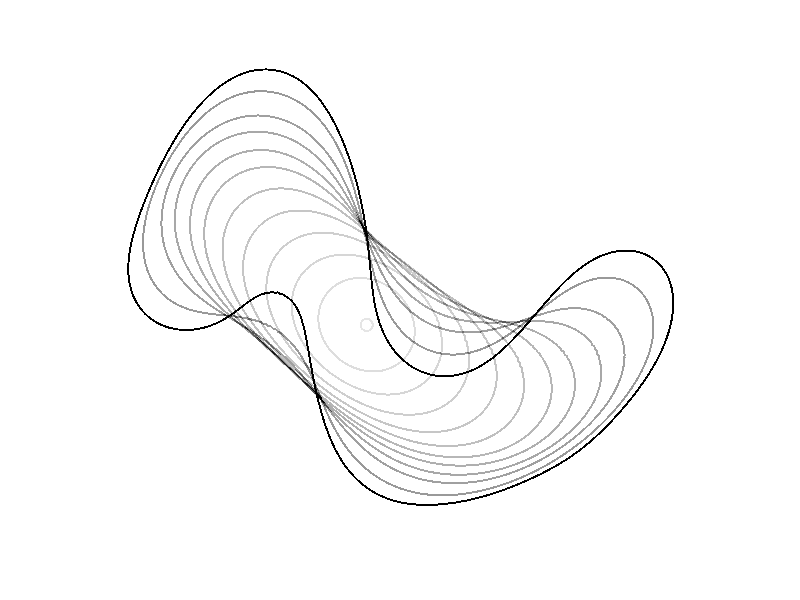}
    \caption{The curve shortening flow is a special case of the mean curvature flow with $n=1$. The initial curve (in black) is shown, with several successive time slices in grey. The curve evolves in time and eventually dies at a round point in finite time. This image was generated using a visualization tool made by Carapetis \cite{Acar}.}
    \label{fig:csfimage}
\end{figure}


\subsubsection{The level set approach}
A very important element to understanding the equivalence we propose is the so-called level set approach to geometric flows. For the mean curvature flow in particular, a solution to the level set equation 
\begin{equation}\label{eq:levelset}
    \begin{split} 
        u_t &= |\nabla u|\;\text{div}\left(\frac{\nabla u}{|\nabla u|}\right) \\
        \Gamma_{0} = &\{x \in \mathbb{R}^{n+1}; u(x,0) = 0\}
    \end{split}
\end{equation} 
will have level sets that satisfy the mean curvature flow in $\mathbb{R}^{n+1}$ with initial data \(\Gamma_{0}\). Moreover, if the domain is restricted to $\Omega$ with $C^2$ boundary, then solutions to equation \ref{eq:levelset} with Neumann boundary conditions will have level sets that satisfy the free boundary mean curvature flow, meaning the surface meets the domain boundary orthogonally for all positive time. Furthermore, the level set flow is known to flow through possible singularities which may form in the classical viewpoint; however, through a process called \emph{fattening} (which is referred to here as the non-empty interior difficulty), this may no longer be unique \cite[Example 7.3]{Ilmanen1992}. If we know some a priori structure on the initial level-set, this may be avoided (see \cite{HershkovitsWhite2020} for example). 

These remarkable results were first shown by Evans and Spruck through a series of papers \cite{Evans1991,Evans1992a,Evans1992b,Evans1995}. Among the important results are the existence and uniqueness of viscosity solutions to \eqref{eq:levelset} and the fact that the motion of the level sets is entirely geometric. Further results concerning the regularity of such solutions may be found in \cite{ColdingMinicozzi2016} and \cite{ColdingMinicozzi2018}. These results, which were initially shown in the Cauchy setting, were extended to work in a bounded domain with $C^2$ boundary and Neumann boundary conditions by \cite{Sato1994}. It is easy to see that with Neumann boundary conditions, a solution to \eqref{eq:levelset} will have level sets that meet the domain boundary orthogonally. This means that the MCF replicated by the level set approach with Neumann boundary conditions is the \textit{free boundary} mean curvature flow, where the evolving surface meets the domain boundary orthogonally. We introduce this approach to understanding mean curvature flow because it is crucial to connect the approximation methods to MCF with the flow itself. 

\subsubsection{Approximation through diffusion}
The most crucial element of the connection between MBR and MCF is a scheme originally used to approximate mean curvature flow developed by Merriman, Bence,  and Osher \cite{Bence1992}. This Scheme (called the MBO scheme) was developed in 1992, and its convergence was proven by Evans in 1993 \cite{Evans1993} and separately by Barles and Georgelin in 1995 \cite{Barles1995}. 

The algorithm they propose is described as follows. Consider a closed region $R_0\subset\mathbb{R}^n$ with a smooth boundary, which we call $\Gamma_0=\partial R_0$. Let $\chi_{R_0}$ be the characteristic function for $R_0$. Find the unique $u$ which solves $u_t=\Delta u\in \mathbb{R}^n\times(0,h)$ with $u(\cdot, 0)=\chi_{R_0}$. This is equivalent to saying $u(\cdot, h)=S(h)\chi_{R_0}$, where $S(h)$ is the semigroup operator for the heat equation.  Having diffused the characteristic function, we form a new region $R_h$ with properties $R_h=\{x\in \mathbb{R}^n;u(x,h)\geq\frac{1}{2}\}$. It was proven by Evans that if $R_0$ is closed with a smooth boundary, so too is $R_h$ \cite{Evans1993}. Repeating this process, we generate a sequence of sets described in the following way

\begin{equation}\label{eq:MBO}
    R_{(n+1)h}=\left\{x\in \mathbb{R}^n; S(h)\chi_{R_{nh}}(x)\geq \frac{1}{2}\right\}
\end{equation}

This process results in a sequence of characteristic functions $\chi_{R_{nh}}$ for the sets $R_{nh}$ which are themselves superlevel sets of the previous characteristic function after diffusing for time $h$. This was originally proposed as a numerical method for the fast computation of curvature flows, but it has been adapted and applied to many different kinds of threshold dynamics and geometric flows.  Convergence was first proved by Evans in 1993 using the theory of nonlinear semigroups to show that the motion generated by one step of the MBO process is within $o(t)$ of the motion of level sets of solutions to \eqref{eq:levelset} after time $t$, and so the limits coincide. At the same time, Barles and Georgelin showed that the limit of the sequence of characteristic functions existed and was both a viscosity sub- and super-solution to \eqref{eq:levelset}.

In either method, they had to contend with the possibility that the level set from the MBO scheme was not guaranteed to have an empty interior, in the same way that the MBR process does not guarantee that the set of players with degenerate best responses will have an empty interior. The conditions for when the non-empty interior difficulty is avoided were made possible by a result about the level set equation by \cite{Barles1993}. For this result, and others to follow, define $dd(x,\partial R)$ as the signed distance of a point $x$ from a surface $\partial R$ which is positive inside of $R$ and negative outside of $R$.
\begin{lemma}[Non-empty interior difficulty for the level set equation (from \cite{Barles1993}]\label{lem:empty interior}
    If $R\subset \mathbb{R}^n$, $\Gamma = \partial R$, and $u$ is the solution to \eqref{eq:levelset} with initial data $u(x,0)=dd(x,\Gamma)$, then $\Gamma_t:=\{u(x,t)=0\}$ has an empty interior if and only if there is a unique solution to \eqref{eq:levelset} with initial data $\chi_R-\chi_{R^c}.$ 
\end{lemma}

Later in their proof of convergence, \cite{Barles1995} showed that these are the exact conditions under which the MBO process will avoid the empty interior difficulty. That is to say, whenever $R$ is such that the solution to \eqref{eq:levelset} with initial data $u(x,0)=dd(x,\partial R)$ in  $\mathbb{R}^n$ has an empty interior, then the MBO process will avoid the non-empty interior difficulty. Although the exact conditions under which this occurs are not perfectly described, \cite{Barles1993} gives a sufficient condition for when the conditions of Lemma \ref{lem:empty interior} are satisfied. For our discussion, we can say that if $\partial R$ is $C^2$ and ``sufficiently well behaved," then it satisfies the condition of Lemma \ref{lem:empty interior} and thus MBO does not encounter the non-empty interior difficulty. For examples of regions which are not ``sufficiently well behaved," see \cite{Evans1991,Evans1992a,Barles1993}.

When the original region is sufficiently well behaved, not only will MBO avoid the non-empty interior difficulty, but the process will also replicate the mean curvature flow as $h\to 0$.  

\begin{lemma}[Convergence of MBO from \cite{Barles1995} (equivalent to \cite{Evans1993})]\label{lem:convergence1}
    Let $R_0\subset \mathbb{R}^n$, and let $u$ be the solution to \eqref{eq:levelset} in $\Omega\times [0,T]$ with initial data $dd(x,\partial R_0)$, If $\Gamma_t=\{u(x,t)=0\}$ has an empty interior (i.e., $R$ is sufficiently well behaved), then 
    \[\lim_{h\to 0}\bigcup_{n=0}^{T/h}(\partial R_{nh}\times \{nh\})=\bigcup_{t\geq 0}(\Gamma_t\times \{t\})\] in the Hausdorff sense. 
\end{lemma}

In the original paper, Bence, Merriman and Osher described how this method could be used even if the original boundary was not simple or if there was a complex of boundaries separating the space into more than two regions. The process for doing this involves simple pairwise comparisons of every diffused region. If a boundary complex separates the space into $N$ regions, each region $R_0^n$ is diffused in the same way so that $u^n(\cdot,h)=S(h)\chi_{R^n_0}$. Having diffused all of the regions, if $u^i(x, h)>u^j(x,h)$ for all $j$, then $x$ is in region $R^i_h$. The new boundary complex will form whenever $u^i(x,h)=u^j(x,h)>u^k(x,h)$  for all $k\neq i, j$, and singular points in the boundary complex appear when there are at least three regions which satisfy 
$u^i(x,h)=u^j(x,h)=u^k(x,h)>u^l(x,h)$ for all $l\neq i,j,k$. Although the standard MCF does not account for such boundary complexes, relaxed versions of the MCF, like the Network flow \cite{Imanen2019}, do.

All of this theory was developed in the Cauchy setting with a Gaussian kernel. Two important generalizations were made. The first came from \cite{Ishii1994}, who found that the kernel need not be Gaussian. The author showed that the same convergence could be observed as long as $K$ satisfies the following conditions 
\begin{enumerate}
    \item [H1] $K$ is non-negative, measurable, and radially symmetric
    \item [H2] $\int_{\mathbb{R}^n}K(z)(1-|z|^2)dz<\infty$
    \item [H3] $\int_{\mathbb{R}^{n-1}}K((w,0))(1-|w|^2)dw< \infty$
    \item [H4] If $R:=(0,1)\to \mathbb{R}$ satisfies $\lim_{\rho \to 0^+}R(\rho)=+\infty$ and $\lim_{\rho \to 0^+}\rho R(\rho)^2=0$, then for any $g(z)=z^\intercal Az+a$ for $A\in \mathcal{S}^{n-1}$ and $a\in \mathbb{R}$, we have
    \[\lim_{\rho \to 0}\sup_{r\in(0,\rho)}\bigg|\int_{B_{n-1}(0,R(\rho))}K((z,rg(z)))g(z)dz-\int_{\mathbb{R}^n}K((z,0))g(z)dz\bigg|=0\]
\end{enumerate}
H1 is meaningful in the game-theoretic context, and its meaning was discussed in Subsection \ref{subsec:coordination}. H2 and H3 are growth conditions on $K(z)$ and a restriction of $K(z)$ to $n-1$ dimensions. The meaning of H4 has no game-theoretic interpretation, and indeed any reasonable kernel used in a game-theoretic context will satisfy H4. One can verify that the standard kernels like the Gaussian, cone, or radial cutoff function satisfy all of these conditions. More information about this assumption can be found in \cite{Ishii1994} and in \cite{Ishii1999}. 

The second important generalization was due to \cite{Ishii2001}, and it showed that this scheme can be adapted to a bounded domain with Neumann boundary data. The result relies on the work of \cite{Sato1994}, which described the level set approach with Neumann boundary conditions. In a bounded domain, the MBO process can be generalized with any kernel that satisfies H1-4, and the domain boundaries found by MBO will still converge to mean curvature flow with a right-angle condition. We now call this the free boundary mean curvature flow. The result is written here, and readers interested in the proof should refer to \cite{Ishii2001}.

\begin{lemma}[Convergence of MBO from \cite{Ishii2001}]\label{lem:convergence2} Let $\Omega$ be a bounded domain with $C^2$ boundary, let the generalize MBO operator in $\Omega$ with parameter $h$ be called 
\begin{equation}\label{eq:generalMBO}
    G_hg(x)=\sup\left\{\lambda \in\mathbb{R};\int_\Omega K^{\sqrt{h}}(x-y)\chi_{g\geq \lambda}(y)dy\geq\frac{1}{2}\int_\Omega K^{\sqrt{h}}(x-y)dy\right\}
\end{equation}
where $K^h(z) = \frac{1}{h^n}K(\frac{z}{h})$ and let $u^h_{nh}(x)$ be the sequence of functions generated by $u_{(n+1)h}^h=G_hu_{nh}^h$. Then for $R_0\subset \Omega$ as in Lemma \ref{lem:empty interior}, 
\[\lim_{h\to 0}\bigcup_{n=0}^{T/h}(u_{nh}^h\times[nh,(n+1)h))\]
is a viscosity solution to \eqref{eq:levelset} in $\Omega\times [0,T]$ with continuous initial data satisfying $\{u(x,0)=0\}=\partial R_0$ and with Neumann Boundary conditions.
\end{lemma}

This is the final result we need to show our main result.  

\section{Main Result}\label{sec:MainResult}
The main result is to show that the MBR sequence converges to free boundary MCF. Our main ingredient will be the equivalence between MBR and generalized MBO. Once we show this, we can simply apply lemmas \ref{lem:convergence1} and \ref{lem:convergence2} to show that MBR is equivalent to MCF. Because the equivalence we show here is not terribly complicated, we will first prove it easily in the Cauchy setting with the Gaussian Kernel for intuition's sake. Then, we will prove it again in the Neumann setting with a general Kernel. 

\begin{theorem}[MBR replicates MCF in $\mathbb{R}^n$]\label{thm:Cauchy}
    Consider a strategy profile $u:\mathbb{R}^n\to \{0,1\}$ so that $R_0=\{x\in \mathbb{R}^n;u(x)=1\}$ is sufficiently well behaved (as in Lemma \ref{lem:empty interior})and let $\Gamma_0=\partial R_0$ be its boundary. Further, let $\cup_{t\geq 0}(\Gamma_t\times\{t\})$ be the \textit{unique} solution to the mean curvature flow from $\Gamma_0$ which exists until some time $T$ which may be infinite. If $\{u(nh)\}_{n=0}^{T/h}$ is the MBR sequence of strategy profiles generated by $\mathcal{H}(h)$ with a Gaussian Kernel $J^{\sqrt{h}}(x-y)=\frac{1}{\sqrt{4\pi h}^n}\exp\left(\left((\frac{|x-y|}{\sqrt{h}}\right)^2\right)$ starting from $\chi_{R_0}$ then 
    \[\lim_{h\to 0}\bigcup_{n=0}^{T/h}(\Gamma_{nh}^h\times \{nh\})=\bigcup_{t\geq 0}(\Gamma_{t}\times \{t\})\]
    In the Hausdorff sense, where $\Gamma_{nh}^h=\partial R_{nh}^h=\partial\{x\in \mathbb{R}^n;u(nh)(x)=1\}$.
\end{theorem}
\begin{proof}
    The first step is to show that a single step of MBR with kernel $J^{\sqrt{h}}$ is equivalent to the MBO process in $\mathbb{R}^n$. This can be done through a simple computation once we observe that $J^{\sqrt{h}}$ is the fundamental solution to the heat equation after time $t=h$. 

    Let $u_0$ be a strategy profile and let $R_0=\{x\in \mathbb{R}^n; u_0(x)=1\}$. Now we compute $\mathcal{H}(h)u_0$ and see that
    \begin{align*}
        \mathcal{H}(h)u_0(x)=1\iff \int_{\mathbb{R}^n}J^{\sqrt{h}}(x-y)\chi_{R_0}(y)dy\geq \int_{\mathbb{R}^n}J^{\sqrt{h}}(x-y)\chi_{R_0^c}(y)dy
    \end{align*}
    Recall that, because $J$ has mass normalized to 1 and that $R_0\cup R^c_0=\mathbb{R}^n$ this is equivalent to writing
    \begin{align*}
        \mathcal{H}(h)u_0(x)=1&\iff \int_{\mathbb{R}^n}J^{\sqrt{h}}(x-y)\chi_{R_0}(y)dy\geq \frac{1}{2}\\
        &\iff (S(h)\chi_{R_0})(x)\geq \frac{1}{2}
    \end{align*}
    This is exactly equivalent to the update rule in the MBO scheme as in \eqref{eq:MBO}, and so we can say that player $x$ plays strategy $1$ if and only if $x\in R_1$, where $R_1$ is the region resulting from one step of MBO starting from $R_0$. This means that the first step is complete and MBR is equivalent to MBO in the Cauchy setting. 

    We have assumed that the initial domain is sufficiently well-behaved that \eqref{eq:levelset} has a solution $v$ with initial data $v(x,0)=dd(x,\partial R_0)$ which has a level set $\{x\in \mathbb{R};v(x)=0\}$ without an interior. By Lemma \ref{lem:empty interior}, there is a unique viscosity solution to \eqref{eq:levelset} with initial data $\chi_{R_0}-\chi_{R_0^c}$, and so the MBO scheme will not develop an interior. 

    Now let $\{R^h_{nh}\}_{n=0}^{T/h}$ be the sequence of regions generated by the MBO scheme with time step $h$. According to \cite{Barles1995} and \cite{Evans1993}, $\lim_{h\to 0}\cup_{n=0}^{T/h} ((\chi_{R_{nh}^h}(x)-\chi_{(R_{nh}^h)^c})\times\{nh\})$ is a viscosity solution for \eqref{eq:levelset} with initial data $\chi_{R_0}-\chi_{R_0^c}$. And so, as in Lemma \ref{lem:convergence1}, 
    \[ \lim_{h\to 0}\bigcup_{n=0}^{T/h} (\partial R_{nh}^h\times\{nh\})=\bigcup_{t\geq 0}\Gamma_t\times\{t\}\]
    in the Hausdorff sense. 

    With both of these points, it is easy to conclude that if $Q_{nh}^h=\{x\in \mathbb{R}^n; u_{nh}(x)=1\}$ and $\Gamma_{nh}^h=\partial Q_{nh}^h$, then $Q_{nh}^h=R_{nh}^h$ for all $n$ and so 
    \[\lim_{h\to 0}\bigcup_{n=0}^{T/h}(\Gamma_{nh}^h\times \{nh\})=\bigcup_{t\geq 0}(\Gamma_{t}\times \{t\})\]
\end{proof}
This main result, in the game-theoretic context, says that the boundaries of the strategic communities evolve according to the mean curvature flow in the limit as time steps become arbitrarily short and players interact with a vanishingly small neighborhood around them. This proof demonstrates this in the Cauchy setting, but we can also demonstrate this in a bounded domain with the more sophisticated machinery of \cite{Ishii1994,Ishii1999,Ishii2001}.

\begin{theorem}[MBR replicates Free Boundary MCF in $\Omega\subset \mathbb{R}^n$.] \label{thm:FreeBoundary}Consider a domain $\Omega\subset\mathbb{R}^n$ which has a $C^2$ boundary. Consider the strategy profile $u_0:\Omega\to \{0,1\}$ so that $R_0=\{x\in \Omega;u_0(x)=1\}$ is $C^2$, compact, and sufficiently well-behaved so that the hypotheses of Lemma \ref{lem:empty interior} are satisfied. Let $\Gamma_0=\partial R_0$. Further, let $\cup_{t\geq 0}(\Gamma_t\times \{t\})$ be the unique solution to the \textit{Free Boundary Mean Curvature Flow} from $\Gamma_0$ on the interval $[0,T]$. If $\{u_{nh}\}_{n=0}^{T/h}$ is the MBR sequence of strategy profiles generated by $\mathcal{H}(h)$ with a kernel satisfying H1-H4 and appropriate scaling as in \eqref{eq:scale} starting from $u_0$, then 
 \[\lim_{h\to 0}\bigcup_{n=0}^{T/h}(\Gamma_{nh}^h\times \{nh\})=\bigcup_{t\geq 0}(\Gamma_{t}\times \{t\})\]
    In the Hausdorff sense, where $\Gamma_{nh}^h=\partial R_{nh}^h=\partial\{x\in \Omega;u_{nh}(x)=1\}$.
    
\end{theorem}
\begin{proof}
    As before, we will start by showing that one step of MBR is equivalent to a single step of the generalized MBO process in $\Omega$. Let $u_0$ be a strategy profile and let $R_0=\{x\in \Omega;u_0(x)=1\}$. Now we compute $\mathcal{H}(h)u_0$ and see that 
    \begin{equation*}
        \mathcal{H}(h)u_0=1\iff \int_\Omega K^{\sqrt{h}}(x-y)\chi_{R_0}(y)dy\geq\int_\Omega K^{\sqrt{h}}(x-y)\chi_{R_0^c}(y)dy 
    \end{equation*}
    Recall that, because relative payoff, not absolute payoff, determines the best response, we can adjust our Kernel to normalize its mass to 1 by dividing by $\|K^{\sqrt{h}}(x-\cdot)\|_{L^1(\Omega)}$. This gives us
    \begin{align*}
        \mathcal{H}(h)u_0(x)=1&\iff \frac{1}{\|K^{\sqrt{h}}(x-\cdot)\|_{L^1(\Omega)}}\int_\Omega K^{\sqrt{h}}(x-y)\chi_{R_0}(y)dy\geq \frac{1}{2}\\
        &\iff \int_\Omega K^{\sqrt{h}}(x-y)\chi_{R_0}(y)dy\geq \frac{1}{2}\int_\Omega K^{\sqrt{h}}(x-y)dy
    \end{align*}
    Now consider $G^h$ as in \eqref{eq:generalMBO}. 
\[G^hg(x)= \sup\left\{\lambda \in\mathbb{R};\int_\Omega K^{\sqrt{h}}(x-y)\chi_{g\geq \lambda}(y)dy\geq\frac{1}{2}\int_\Omega K^{\sqrt{h}}(x-y)dy\right\}.\]
and let $g(x)$ be any continuous function so that $g(x)\geq \frac{1}{2}$ for $x\in R_0$ and $g<\frac{1}{2}$ for $x\notin R_0$. Such a function always exists (e.g., $g(x)=dd(x,\partial R_0)+\frac{1}{2}$). Now observe that 
\begin{equation*}
    \begin{split}
        G^hg(x)\geq \frac{1}{2}&\iff\int_\Omega K^{\sqrt{h}}(x-y)\chi_{g\geq \frac{1}{2}}\geq\frac{1}{2}\int_\Omega K^{\sqrt{h}}(x-y)dy\\
        &\iff \int_\Omega K^{\sqrt{h}}(x-y)\chi_{R_0}\geq\frac{1}{2}\int_\Omega K^{\sqrt{h}}(x-y)dy\\
        &\iff \mathcal{H}(h)u_0(x)=1
    \end{split}
\end{equation*}

This means that if $g$ is any continuous function so that $g(x)\geq \frac{1}{2}\iff u(x)=1$ then $G^hg(x)\geq \frac{1}{2}\iff \mathcal{H}(h)u(x)=1$. Thus, by simple induction, if $u_{nh}$ is a sequence of strategy profiles generated by $\mathcal{H}(h)$ starting from $u_0$ and $g_{nh}$ is a sequence of continuous functions generated by $G^h$ starting from $g_0$, and if $g_0(x)\geq \frac{1}{2}\iff u_0(x)=1$, then $g_{nh}(x)\geq \frac{1}{2}\iff u_{nh}(x)=1$ for all $n$. Importantly, this implies that, as long as the level set $\{g_{nh}(x)=\frac{1}{2}\}$ does not develop an interior, then 
\[\left\{g_{nh}(x)=\frac{1}{2}\right\}=\partial\left\{u_{nh}(x)=1\right\}=\Gamma^h_{nh}\]

Recall our assumption that $u_0$ is such that $R_0=\{u_0=1\}$ is sufficiently well-behaved so that the solution to \eqref{eq:levelset} with initial data $g_0(x)=dd(x,\partial R)$ will not develop an interior as in Lemma \ref{lem:empty interior}. With this assumption, we need only note that, by Lemma \ref{lem:convergence2}, it is certain that the sequence $g_{nh}(x)$ generated by $G^h$ converges to a viscosity solution of \eqref{eq:levelset}. Namely,
\[\lim_{h\to 0}\bigcup_{n=0}^{T/h}(g_{nh}^h\times[nh,(n+1)h):=g\]
That viscosity solution, $g$, will not encounter the empty interior difficulty. Moreover, under these conditions, the strategic boundary $\Gamma_{nh}$ will also not develop an interior because it will conform to the level sets $\{g_{nh}(x)=\frac{1}{2}\}$ in the generalized MBO process, by the computation above. 

Now we conclude by saying that the level set of the viscosity solution $\{g(x,t)=\frac{1}{2}\}$ satisfies the free boundary mean curvature flow starting from $\{g(x,0)=\frac{1}{2}\}=\partial R_0=\Gamma_0$, which we call $\cup_{t\geq 0}(\Gamma_t\times\{t\})$. Thus, we have shown 
\begin{align*}
    \lim_{h\to 0}\bigcup_{n=0}^{T/h}(\Gamma_{nh}^h\times\{nh\})&=\lim_{h\to 0}\bigcup_{n=0}^{T/h}\left(\partial \left\{g_{nh}^h(x)=\frac{1}{2}\right\}\times[nh,(n+1)h)\right)\\
    &=\bigcup_{t\geq 0}\left(\left\{g(x,t)=\frac{1}{2}\right\}\times \{t\}\right)\\
    &=\bigcup_{t\geq 0}(\Gamma_t\times\{t\})
\end{align*}
The first equality is convergence in the Hausdorff sense. Note that, if $\Gamma_{nh}^h=\partial\{g_{nh}^h(x)=\frac{1}{2}\}$, then the Hausdorff distance between $\Gamma_{nh}^h\times\{nh\}$ and$\partial\{g_{nh}^h(x)=\frac{1}{2}\}\times [nh,(n+1)h)$ is exactly $h$ for any $n$. The second equality is the result of \cite{Ishii2001}, and the third equality comes from the geometric quality of solutions to \eqref{eq:levelset} as described by \cite{Evans1991}.
\end{proof}

This complete result shows that the strategic boundary in the continuous coordination game evolves under MBR (in a certain limiting sense) by Free Boundary Mean Curvature Flow. On its own, the result is interesting enough. From a geometer's perspective, we have constructed an elementary game and, from game-theoretic first principles, can replicate a fundamental geometric flow. There is a history of finding such game-theoretic equivalences (e.g. \cite{Gonzalvez2025,Kohn2006}), but ours is the simplest game in this category of which we are aware. Of course, the results that such games are equivalent to mean curvature flow imply that all the games themselves are equivalent in some sense. 

From a game-theoretic perspective, however, the equivalence between this fundamental game and this geometric flow allows us to answer game-theoretic questions with geometric answers. A discussion of such questions and answers follows in section \ref{sec:Implications}.

There are two points related to the main results that we must still address. The first is that having an empty interior does not necessarily mean the boundary has measure 0, so the strategic boundary may have an empty interior, but the set of players with degenerate best responses may not have measure zero. Although we conjecture it is possible to argue that the strategic boundary of measure zero is implied by properties of the strategic communities under these evolutions, it is simpler to note that every result listed and every proof given work identically for a strategy profile $u$ and its opposite $1-u$. The decision to say that $\mathcal{H}(h)u(x)=1\iff K^\epsilon *u\geq \frac{1}{2}$ instead of using a strict inequality is equivalent to giving a tie-breaking order where, if strategies $0$ and $1$ are both best responses, pick strategy $1$. Because we do the same process on $1-u$ and get the same evolution of the strategic boundary, we have shown that any tie-breaking order works equally well so long as the initial strategic boundary is sufficiently well-behaved. 

We must also point out that MBO and even the generalization of MBO can be extended to describe the evolution of multiple region boundaries by doing pairwise comparison of each pair of regions according to equations \eqref{eq:MBO} or \eqref{eq:generalMBO}. As an example, consider three pairwise disjoint regions $R_1, R_2$ and $R_3=\Omega \setminus R_1\setminus R_2$. With each of these is an associated function $g_1,g_2,$ and $g_3$ so that $g_i\geq \frac{1}{2}\iff x\in R_i$ for $i=1,2,3$. To extend the generalized MBO scheme to this setting, we would say that 
\begin{equation*}
    G^hg_i(x)=\min_{j\neq i}\sup\left\{\lambda \in \mathbb{R}; \int_\Omega K^{\sqrt{h}}(x-y)\chi_{g_i\geq \lambda}(y)dy \geq \int_\Omega K^{\sqrt{h}}(x-y)\chi_{g_j\geq \lambda}(y)dy\right\}
\end{equation*}

In this setting, it is even easier to see the equivalence between this and the best response function because it is a direct comparison between functions of the form $K*\chi_{R^i}$ and $K*\chi_{R^j}$. We can again show the equivalence as in the proof of \ref{thm:FreeBoundary}. Suppose $R_1,R_2$ and $R_3$ fill $\Omega$ and let $g_i(x)$ be continuous functions such that $g_i(x)\geq \frac{1}{2}\iff x\in R_i$. Now note that $G^hg_i(x)\geq \frac{1}{2}$ is equivalent to saying that for all $j \neq i$
\[\sup\left\{\lambda \in \mathbb{R}; \int_\Omega K^{\sqrt{h}}(x-y)\chi_{g_i\geq \lambda}(y)dy \geq \int_\Omega K^{\sqrt{h}}(x-y)\chi_{g_j\geq \lambda}(y)dy\right\}\geq \frac{1}{2}\]
This is again equivalent to saying that 
\[ \int_\Omega K^{\sqrt{h}}(x-y)\chi_{g_i\geq \lambda}(y)dy \geq \int_\Omega K^{\sqrt{h}}(x-y)\chi_{g_j\geq \lambda}(y)dy\quad \forall j\neq i\]
Now we recall that these convolutions are exactly the expressions for payoff, so we have shown that 
\[G^hg_i(x)\geq\frac{1}{2}\iff w(x,i|u)\geq w(x,j|u) \  \forall j\neq i \iff \mathcal{H}(h)u(x)=i\]

Thus we have shown that if $R_i$ are regions which fill $\Omega$ and $g_i$ are such that $g_i(x)\geq \frac{1}{2}\iff x\in R_i$ then $G^hg_i(x)\geq \frac{1}{2}\iff x\in \{x\in \Omega, \mathcal{H}(h)u(x)=i\}$. This is a slightly more complicated way of saying that the strategic boundaries of strategic community $i$, under $\mathcal{H}(h)$, will coincide with the level sets of $\{G^hg_i(x)=\frac{1}{2}\}$ and thus, even in the more general setting, MBR is equivalent to the generalized MBO scheme. We cannot at present use this to form an equivalence between MBR and a more general notion of mean curvature flow which allows for these boundary complexes, but we can show empirically that the strategic boundaries seem to move according to the free boundary network flow (e.g., Fig. \ref{fig:network}). In the figure, the strategic boundaries (the curves that separate the green, yellow, and cyan regions) seem to move according to MCF away from the junctions. This motion is similar to the Network flow of \cite{Imanen2019} except for the fact that curves in the simulation can meet and annihilate each other or join together. In the actual limit, this would not happen because curves only show this behavior when they come within $\sqrt{h}$ of each other and otherwise cannot coincide in finite time because of the avoidance principle. 
\begin{figure}
    \centering
    \includegraphics[width=0.8\linewidth]{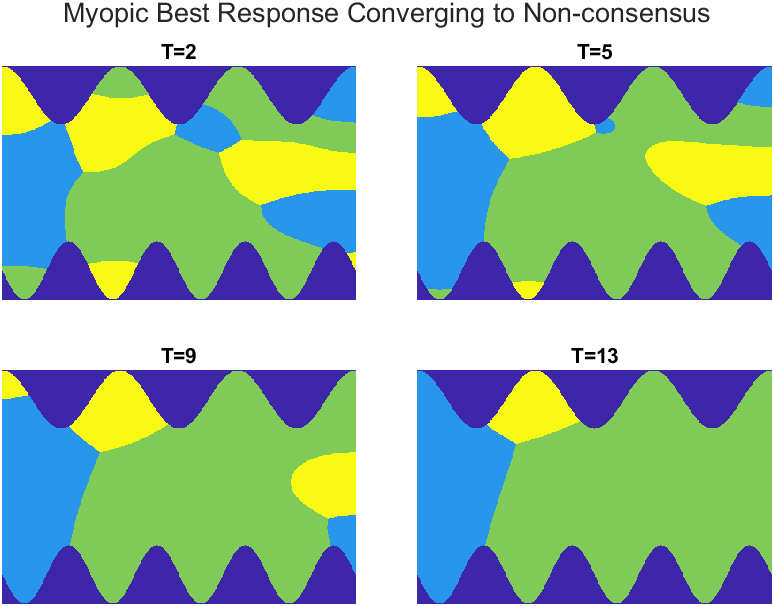}
    \caption{A strategy profile with three strategies evolving under MBR. }
    \label{fig:network}
\end{figure}

\section{Implications}\label{sec:Implications}

The most obvious implication of the main result is that Nash equilibria are minimal surfaces, which are stationary solutions to MCF or free-boundary MCF. This gives us an opportunity to do a complete classification of Nash equilibria given a particular domain. For instance, if the domain is a disk in \(\mathbb{R}^{2}\), the only minimal surfaces are diameters, and thus diameters are the only strategic boundaries that correspond to Nash equilibria. In a more general 2D domain which is strictly convex, still the only Nash equilibria correspond to diameters where they exist. These examples are rather trivial, but they show a general principle which is that from the geometry of the player domain alone, we can determine the kinds of Nash equilibria that are admitted.  

Remaining in the 2D case, we can use some results well understood from the theory of curve shortening flow to describe the dynamics of strategic communities under coordination. For instance, in curve shortening flow, the area enclosed by a curve changes at a rate proportional to the turning angle of the curve. That is, if a curve entirely encloses an area, the turning angle of the curve is one complete rotation ($2\pi$) and the enclosed area will shrink at a rate of $2\pi$ units of area per time. This allows one to exactly compute the time of extinction for a strategic community based only on its size and its intersection with the domain boundary. An interesting point about this is that, if a strategic community is entirely surrounded by another strategic community, under MBR the surrounded strategic community will disappear after a length of time which depends only on the initial area and not the initial shape. 

A last point to this same effect is that if a domain has at least three strategic communities but no two strategic boundaries intersect initially, then no two strategic boundaries will ever intersect, at least in the limiting process as $h\to 0$. This result is equivalent to the well-known avoidance principle for curve shortening flow. If no two strategic boundaries intersect, then the results above still hold for each of the strategic boundaries individually, and so each will move according to the mean curvature flow. Because two surfaces evolving under MCF that do not initially intersect cannot intersect at any future time, we know that the strategic boundaries will never intersect. This, of course, is only valid in the limit as $h\to 0$ because (as seen in Figure \ref{fig:network}) for any positive $h$, if strategic boundaries come within $\sqrt{h}$ of one another, they may combine into a single strategic boundary or annihilate one another and leave behind no strategic boundary.  If boundary 1 separates strategy $a$ from strategy $b$ and boundary 2 separates strategy $b$ from strategy $c$, the contact of these boundaries will result in a single boundary between strategy $a$ and $c$. However, if boundary 1 separates strategy $a$ from $b$ and boundary 2 separates strategy $b$ from $a$ in the opposite order, the contact of these curves results in no boundary. The fact that these behaviors depend on the actual strategy profiles and not just the geometry of the strategic boundary is proof that without the limit, this game cannot be understood entirely geometrically. However, even for positive $h$, as long as boundaries are at least a distance of $\sqrt{h}$ away from one another, the behavior of the boundary curves will at least approximate the curve shortening flow.

It may be of interest to some, especially those studying applications of dynamic game theory in ecology or economics, to consider the stability of Nash equilibria. To borrow from the evolutionary ecological vocabulary, a convergent stable equilibrium (CS-equilibrium) is a Nash equilibrium, $u$, such that all nearby strategy profiles will converge to $u$ under a certain strategy revision protocol, which, in our case, is MBR. This is similar to the concept of a Convergence Stable Strategy, a refinement of the more well-known Evolutionary Stable Strategy (ESS),\cite{Apaloo2009,MaynardSmith1973}. However, the spatially explicit nature of the strategy profiles in question here makes the terms CSS and ESS inappropriate. These CS-equilibria are the only relevant equilibria in many application settings because a system that is at a Nash equilibrium but not a CS-equilibrium will fall away from equilibrium after a small perturbation. CS-equilibria, in our case, correspond to strategic boundaries which, if perturbed by some small distance in $C^2$, will return to the original strategic boundary under mean curvature flow. 

This equivalence means that, in a 2D convex domain, there are no Convergent stable Equilibria. Every Nash equilibrium can be perturbed in such a way that the strategic boundary moves away from the original strategic boundary. Indeed, every Nash equilibrium can be perturbed in such a way that one strategic community goes extinct in finite time. For this reason, stable coexistence is only possible in non-convex domains. 

Lastly, we look ahead to future work and consider what this system might look like if one strategy offered a better payoff than the other (e.g., Figure \ref{fig:bias}). 
\begin{figure}
    \centering
    \includegraphics[width=0.7\linewidth]{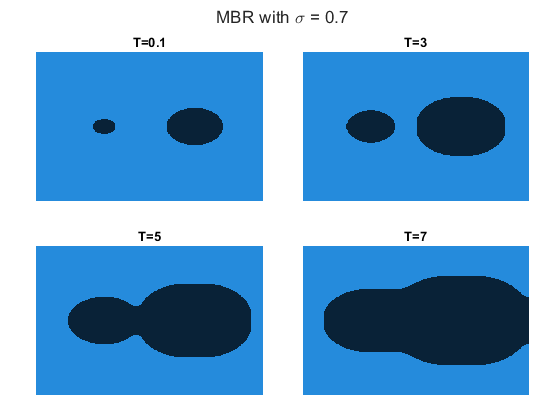}
    \caption{Myopic best response in the setting where strategy 1 (in dark blue) is favored over strategy 2 (in light blue). The fronts move outward at a constant speed until a singularity forms, at which point the strategic boundary smooths in a manner like curve shortening flow.}
    \label{fig:bias}
\end{figure} 
In the system described in section \ref{sec:modelBackground}, the payoff matrix for each pairwise interaction is the identity matrix. However, if instead the payoff matrix were a diagonal matrix with entries $\sigma_1,...,\sigma_m$, players may adopt a new strategy even if it is not the most popular strategy in its neighborhood. In the simplest case with two strategies, let the payoff matrix be
\[\begin{bmatrix}
    1&0\\0&\sigma
\end{bmatrix}\]
The example in Figure \ref{fig:bias} shows strategy 1 in dark blue with a payoff of $1$ and strategy 2 in light blue with a payoff of $\sigma = 0.7$. 

If we try to examine the geometric properties of this flow through the same equivalence as before, we can see that MBR is equivalent to the following MBO-type operation
\begin{align*}\tilde G^hg_1(x)&=\sup\left\{\lambda\in \mathbb{R};\int_{\Omega}K^{\sqrt{h}}(x-y)\chi_{g\geq\lambda}(y)dy\geq \frac{1}{1+\sigma}\int_\Omega K^{\sqrt{h}}(x-y)dy\right\}
\end{align*}
Sequences generated by $\tilde{G}^h$ do not converge to solutions of \eqref{eq:levelset} as $h\to 0$, and so it is certain that the motion of the strategic boundaries in this case will not move according to mean curvature. Worse still, numerical experiments suggest that the motion of the boundaries would depend on $\epsilon$ in such a way that no limit exists (see figures \ref{fig:convergence1} and \ref{fig:convergence2}). When the curvature of a strategic boundary is far smaller than $\frac{1}{\sqrt{h}}$ (as is always the case in the limit), the strategic boundary moves with approximately constant speed, which depends on $\epsilon$ in the direction of the less favorable strategy. 
\begin{figure}[ht]
    \centering
    \includegraphics[width=0.8\linewidth]{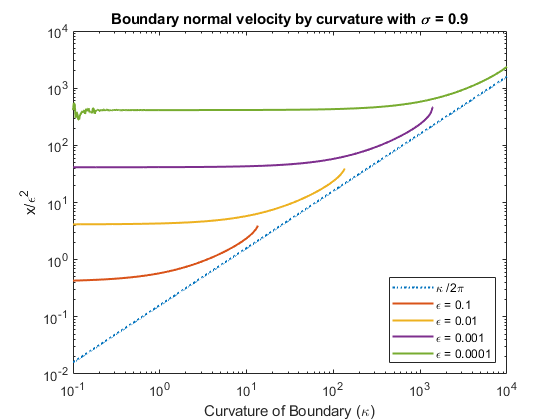}
    \caption{The numerically approximated speed of the boundary depending on the curvature of the boundary. When $\sigma<1$, strategy 1 is more favorable. Boundary curves always move faster in the outward direction than they would in the case where $\sigma = 1$ (marked in the dotted blue line). }
    \label{fig:convergence1}
\end{figure}
\begin{figure}[ht]
\centering
    \includegraphics[width=0.8\linewidth]{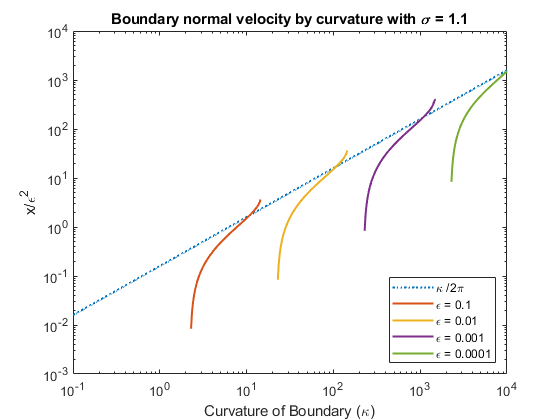}
    \caption{The numerically approximated speed of the boundary depending on the curvature of the boundary. When $\sigma>1$, strategy 1 is less favorable. The boundary curve moves more slowly in the outward direction than it would in the case where $\sigma =1$. It may even move in the inward direction even though the curvature vector points outward.}
    \label{fig:convergence2}
\end{figure}
This falls far short of a proof, but it gives us the numerical suggestion that the geometric properties of coordination are not generic. From the numerical experiments like those depicted in Figure \ref{fig:bias}, it appears as though the strategic boundaries move in a way more similar to front propagation with constant outward velocity (see Figure \ref{fig:convergence1}); only in the places where singularities form or where the strategic boundary forms a sharp angle with the domain boundary does motion by mean curvature describe the flow. Readers who are interested in simulating these types of flows can access them in the repository linked in the code availability statement of this manuscript, which is cited here as \cite{McAlister2026Code}. Further research must be done to understand the ways that the geometric properties of the neutral coordination game can be extended to a more general class of games.

Although this manuscript does not present a great technical achievement (the only real innovation is the equivalence between equations \eqref{eq:mbr},\eqref{eq:MBO}, and \eqref{eq:generalMBO}), it has shown a great connection between two fundamental elements of two very separate fields of mathematics. This result allows us to use results from one to guide inquiry in the other. Most crucially, this connection allows us to describe coordination (at least in a certain limit) through geometric properties alone, which helps us understand the relationship between relational structure and game-theoretic outcome.  

\section*{Acknowledgments} This work was done with advice from Dr. Tadele Mengesha, Dr. Nina H. Fefferman, and Dr. Theodora Bourni.
\section*{Declarations}
\subsection*{Funding} While completing this work, JSM was supported by NSF grants DBI
2412115 and DBI 2622265 as part of the US NSF Center for Analysis and
Prediction of Pandemic Expansion (APPEX).

\subsection*{Conflict of interest/competing interests} The authors have no competing interests to declare

\subsection*{Code availability} The MATLAB code used to simulate myopic best response, generate figures, and examine convergence numerically is included in the following repository linked below and cited as \cite{McAlister2026Code}.

\noindent \href{https://github.com/jmcalis/JSM_2026_DualCoordination/releases/tag/v1.1.0}{https://github.com/jmcalis/JSM\_2026\_DualCoordination/releases/tag/v1.1.0} 

\subsection*{Author contribution}
Both authors contributed to the study conception and analysis. Simulations and visualizations were written and performed by John S. McAlister. Both authors contributed to the writing of the initial draft.

\bibliographystyle{alpha}
\bibliography{refs.bib}
\end{document}